\documentclass{article}

\usepackage{a4wide}
\usepackage{amsmath,amssymb,amsthm}
\usepackage{graphicx}
\usepackage{hyperref}
\usepackage{url}
\usepackage{xcolor}

\newcounter{theorem}
\newtheorem{Theorem}[theorem]{Theorem}

\newcounter{corollary}
\newcounter{proposition}
\newcounter{lemma}
\newcounter{example}
\newcounter{remark}
\newtheorem{Remark}[remark]{Remark}

\newcounter{definition}
\newtheorem{Definition}[definition]{Definition}

\newcommand\blfootnote[1]{%
	\begingroup
	\renewcommand\thefootnote{}\footnote{#1}%
	\addtocounter{footnote}{-1}%
	\endgroup
}

\begin{document}

\title{A Necessary Optimality Condition for Extended Weighted 
Generalized Fractional Optimal Control Problems\blfootnote{This 
is a preprint version of the paper published open access in 
'Results in Control and Optimization' [https://doi.org/10.1016/j.rico.2023.100356].}}

\author{Houssine Zine$^{1}$\\       
\texttt{zinehoussine@ua.pt}
\and El Mehdi Lotfi$^{2}$\\         
\texttt{elmehdi.lotfi@univh2c.ma}
\and Delfim F. M. Torres$^{1,}$\thanks{Corresponding author.}\\
\texttt{delfim@ua.pt}              
\and Noura Yousfi$^{2}$\\          
\texttt{nourayousfi.fsb@gmail.com}}

\date{\mbox{$^{1}$Center for Research and Development in Mathematics and Applications (CIDMA),}
Department of Mathematics, University of Aveiro, 3810-193 Aveiro, Portugal\\[0.3cm]
$^{2}$Laboratory of Analysis, Modeling and Simulation (LAMS),\\
Faculty of Sciences Ben M'sik, Hassan II University of Casablanca,\\
P.O. Box 7955 Sidi Othman, Casablanca, Morocco}


\maketitle

\begin{abstract}
Using the recent weighted generalized fractional
order operators of Hattaf, a general fractional optimal control
problem without constraints on the values of the control functions
is formulated and a corresponding (weak) version of
Pontryagin's maximum principle is proved.
As corollaries, necessary
optimality conditions for Caputo--Fabrizio,
Atangana--Baleanu and weighted Atangana--Baleanu
fractional dynamic optimization problems
are trivially obtained. As an application,
the weighted generalized fractional problem
of the calculus of variations is investigated
and a new more general fractional Euler--Lagrange
equation is given.

\medskip


\medskip

\end{abstract}


\section{Introduction}
\label{sec:01}

In 1744, Euler published his ``Methodus inveniendi lineas curvas maximi
minimive proprietate gaudentes sive solutio problematis
isoperimetrici latissimo sensu accepti'' \cite{Euler}. In this work,
Euler seeks a method to find a curve that minimizes or maximizes any quantity
expressed by an integral, generalizing the problems studied before by the
Bernoulli brothers', but retaining the geometrical approach
developed by Johann Bernoulli to solve them. Euler found what is
now known as the Euler--Lagrange differential
equation, which is a necessary condition for a function to maximize or minimize
a given integral functional that depends on the unknown function
and its derivative. In 1760, Lagrange published an ``Essay on a new method
of determining the maxima and minima of indefinite integral formulas'' \cite{Lagrange},
giving the analytic method that allows to attack all types of variational problems
\cite{Pollard,MR3822307}.

In 1996, Riewe introduced the fractional calculus of variations \cite{MR1401316,Riewe},
from which many scholars reintroduced several variants of the fractional calculus of variations
and obtained some wonderful and useful fractional Euler--Lagrange
equations with some additional results: see, e.g., the works of
Agrawal \cite{Agrawal},
Almeida et al. \cite{Almeida},
Atanackovic et al. \cite{Atanackovic1,Atanackovic2},
Frederico and Torres \cite{Frederico},
Klimek \cite{Klimek},
Zine et al. \cite{Zine,integ},
just to mention a few relevant names and works.

At the end of the fifties of the 20th century, Pontryagin
and his collaborators formulated one of the most important tools
to characterize the optimal trajectories of a controllable system
\cite{Pontryagin}. In the particular case when the controls
are continuous functions without constraints on its values,
the necessary optimality conditions given by Pontryagin's
maximum principle (PMP) reduce to the classical Euler--Lagrange equations.
This is known as the weak PMP \cite{MyID:489}, in contrast with the
general (strong) PMP \cite{MyID:487}.

Fractional optimal control problems (FOCPs)
give an extension of the classical optimal control problems
(see, e.g., \cite{Kamocki}),
where not only the control system is described by some fractional
differential equation \cite{Tricaud}, but the cost functional may
be also provided by a fractional integral operator \cite{Ali}.

Motivated by the new generalized fractional operators explored by Hattaf \cite{hattaf} and
the weighted generalized fractional integration by parts formula proved by Zine et al. \cite{integ}
(see Section~\ref{sec:02}), together with the urgent need to use advanced techniques to quantify
the optimal controls for the considered optimization problems, we prove in Section~\ref{sec:03}
a weak version of a weighted generalized fractional Pontryagin's maximum principle,
supported by an adequate application to the calculus of variations (Section~\ref{sec:04}).
We end with Section~\ref{sec:06} of conclusions and possible future work.


\section{Preliminaries}
\label{sec:02}

In this section, we recall the necessary mathematical prerequisites
to formulate and prove our main theorem, which gives necessary optimality
conditions of Pontryagin type in the weighted generalized fractional setting.
We adopt here the following notations:
\begin{gather*}
\phi(\alpha):=\dfrac{1-\alpha}{B(\alpha)},
\quad \psi(\alpha):=\dfrac{\alpha}{B(\alpha)},
\quad \mu_\alpha:=\dfrac{\alpha}{1-\alpha},
\end{gather*}
where $0\leq\alpha< 1$ and $ B(\alpha)$ is a normalization function
obeying $B(0)=B(1)=1$.

Let $H^1(a, b)$ be the Sobolev space of order one defined as follows:
$$
H^1(a, b)=\left\{y\in L^2(a,b) : y'\in L^2(a,b) \right\}.
$$

\begin{Definition}[The left-weighted generalized fractional derivative and integral \cite{hattaf}]
Let $0\leq\alpha< 1$, $\beta>0$ and $f\in H^1(a, b)$. 
The left-weighted generalized fractional derivative of order $\alpha$ 
of function $f$, in the Riemann--Liouville sense, is defined by
\begin{equation}
\label{NGDR:g}
\,_{a,w}^{R}D^{\alpha,\beta}f(x)
= \dfrac{1}{\phi(\alpha)}\dfrac{1}{w(x)}\dfrac{d}{dx}\int^{x}_{a}(wf)(s)
E_\beta\left[ -\mu_{\alpha}(x-s)^{\beta}\right]ds,
\end{equation}
where $E_{\beta}$ denotes the Mittag--Leffler
function of parameter $\beta$ defined by
\begin{equation}
\label{mitag}
E_{\beta}(z)=\sum_{j=0}^{\infty}\dfrac{z^{j}}{\Gamma(\beta j+1)},
\end{equation}
and $w\in C^{1}([a,b])$ with $ w>0$.
The corresponding left fractional integral is defined by
\begin{equation}
\label{NGI:g}
_{a,w}I^{\alpha,\beta}f(x)
= \phi(\alpha)f(x)+\psi(\alpha)\,^{RL}_{a,w}I^{\beta}f(x),
\end{equation}
where $^{RL}_{a,w}I^{\beta}$ is the standard left weighted Riemann--Liouville
fractional integral of order $\beta$ given by
\begin{equation}
^{RL}_{a,w}I^{\beta}f(x)=\dfrac{1}{\Gamma(\beta)}
\dfrac{1}{w(x)}\int^{x}_{a}(x-s)^{\beta-1}w(s)f(s)ds,
\quad x > a.
\end{equation}
\end{Definition}

\begin{Definition}[The right-weighted generalized fractional derivative and integral \cite{integ}]
\label{NGDR:d}
Let $0\leq\alpha< 1$, $\beta>0$ and $f\in H^1(a, b)$. 
The right-weighted generalized fractional derivative
of order $\alpha$ of function $f$, in the Riemann--Liouville sense, is defined by
\begin{equation}
^{R}D^{\alpha,\beta}_{b,w}f(x)
= \dfrac{-1}{\phi(\alpha)}\dfrac{1}{w(x)}\dfrac{d}{dx}\int^{b}_{x}(wf)(s)
E_\beta\left[ -\mu_{\alpha}(s-x)^{\beta}\right]ds,
\end{equation}
where $w\in C^{1}([a,b])$ with $w>0$.
The corresponding fractional integral is defined by
\begin{equation}
\label{NGI:d}
I^{\alpha,\beta}_{b,w}f(x)
= \phi(\alpha)f(x)+\psi(\alpha)\,^{RL}I^{\beta}_{b,w}f(x),
\end{equation}
where $^{RL}I_{b,w}^{\beta}$ is the standard right weighted Riemann--Liouville
fractional integral of order $\beta$:
 \begin{equation}
^{RL}I_{b,w}^{\beta}f(x)=\dfrac{1}{\Gamma(\beta)}
\dfrac{1}{w(x)}\int^{b}_{x}(x-s)^{\beta-1}w(s)f(s)ds,
\quad x < b.
\end{equation}
\end{Definition}

\begin{Theorem}[The weighted generalized formula of integration by parts \cite{integ}]
\label{GIBP}
Let $f$, $g$ $\in$ $H^1(a,b)$ and $\omega \in C^1(a,b)$ with $\omega>0$.
If $0 \leq \alpha < 1$ and $\beta >0$, then
\begin{eqnarray}
\label{RDaw:8}	
\int_{a}^{b}f(x)\left(\,^RD^{\alpha,\beta}_{b,w}g\right)(x)dx
&=& \int_{a}^{b} w(x)^2g(x)\left(\,_{a,w}^RD^{\alpha,\beta}
\left(\dfrac{f}{w^2}\right)\right)(x)dx,\\
\int_{a}^{b}f(x)(\,^R_{a,w}D^{\alpha,\beta}g)(x)dx
&=& \int_{a}^{b} w(x)^2g(x)\left(\,^RD^{\alpha,\beta}_{b,w}
\left(\dfrac{f}{w^2}\right)\right)(x)dx.\label{RDaw}
\end{eqnarray}
\end{Theorem}

\begin{Remark}
It should be noted that the fractional integration by parts formulas
\eqref{RDaw:8} and \eqref{RDaw} are not always true: 
they only hold under certain conditions. In general, 
$f$ must be in $L^p$ while the second function $g$ must be in $L^q$ 
in such a way $1/p+1/q=1$. In our work these conditions are verified. 
Indeed, as we shall see, the state variable of our control system is in $H^1$, 
which means that it is in $L^2$, and the costate variable $\lambda$ 
is an absolutely continues function on a bounded and closed interval $[0,T]$, 
which means that it is also a $L^2$ function.
\end{Remark}


\section{Main Result}
\label{sec:03}

In this section we formulate an extended weighted generalized
fractional optimal control problem (EWGFOCP)
and derive necessary conditions for the optimality of the EWGFOCP.
Let $U \subset \mathbb{R}$ be a given nonempty open set. 
A control $u :$ $[0,T]$ $\longrightarrow$ $U$ is called
admissible if it is continuous on $[0,T]$. 
The set of all admissible controls is denoted by $U_{ad}$.
The problem consists to determine the optimal control $u(\cdot)$
and corresponding state trajectory $x(\cdot)$
that minimizes a given cost functional performance $J$,
\begin{equation}
\label{P1}
J[x,u]=\int_{a}^{b}L\left(t,x(t),u(t)\right)dt\longrightarrow \min,
\end{equation}
subject to a given dynamic control system,
\begin{equation}
\label{P11}
\,^{R}_{a,w}D^{\alpha,\beta}x(t)
=f\left(t,x(t),u(t)\right),
\quad t\in[a,b],
\end{equation}
assumed to be controllable, and given boundary conditions
\begin{equation}
\label{P111}
x(a)=x_a, \quad x(b)=x_b,
\end{equation}
where $L, f : [a,b]\times \mathbb{R}\times \mathbb{R} \rightarrow \mathbb{R}$
are differentiable functions with respect to all their arguments.

\begin{Theorem}[Necessary optimality condition for the EWGFOCP (\ref{P1})--(\ref{P111})]
\label{OC}
Let $0 \leq \alpha < 1$, $\beta >0$, and
$\left(x(\cdot),u(\cdot)\right)$ be a minimizer of problem
(\ref{P1})--(\ref{P111}). Then, there exists a function
$\lambda(\cdot)$ for which the triplet
$\left(x(\cdot),\lambda(\cdot),u(\cdot)\right)$ satisfies:
\begin{itemize}
\item[(i)] the Hamiltonian system
\begin{equation}
\label{eq:HS}
\left\{
\begin{array}{ll}
\,^{R}_{a,w}D^{\alpha,\beta}x(t)=\dfrac{\partial \mathcal{H}}{\partial
\lambda}\left(t, x(t),\lambda (t),u(t)\right), & t\in[a,b],\\[0.3cm]
w(t)^2\,^{R}D_{b,w}^{\alpha,\beta}\dfrac{\lambda(t)}{w(t)^2}
=\dfrac{\partial \mathcal{H}}{\partial x}\left(t, x(t),\lambda (t),u(t)\right),
& t\in[a,b];
\end{array}
\right.
\end{equation}
\item[(ii)] and the stationary condition
\begin{equation}
\label{eq:SC}
\dfrac{\partial \mathcal{H}}{\partial u}\left(t, x(t),\lambda (t),u(t)\right)=0,
\quad t\in[a,b],
\end{equation}
\end{itemize}
where $\mathcal{H}$ is a scalar function, called the Hamiltonian, defined by
\begin{equation}
\label{P3}
\mathcal{H}\left(t, x,\lambda,u\right)
= L\left(t,x,u\right)
+ \lambda f\left(t,x,u\right).
\end{equation}
\end{Theorem}

\begin{proof}
We consider the following augmented
functional of the performance index:
\begin{equation}
\label{aug}
J_a[x,\lambda,u]=\int_{a}^{b}\left[\mathcal{H}\left(t, x(t),\lambda (t),u(t)\right)
-\lambda (t) f\left(t,x(t),u(t)\right)\right]dt.
\end{equation}
Regarding the variation of the augmented performance index (\ref{aug}),
\begin{equation}
\label{ogmented}
\delta J_{a}[x,\lambda,u]=\int_{a}^{b}\left\{\dfrac{\partial
\mathcal{H}}{\partial x}\delta x(t)
+ \left[\dfrac{\partial \mathcal{H}}{\partial \lambda}
- \,^{R}_{a,w}D^{\alpha,\beta}x(t)\right]\delta \lambda(t)
+ \dfrac{\partial \mathcal{H}}{\partial u}\delta u(t)
- \lambda(t) \,^{R}_{a,w}D^{\alpha,\beta}\delta x(t) \right\}dt.
\end{equation}
Using the weighted generalized fractional integration
by parts formula \eqref{RDaw} of Theorem~\ref{GIBP}, we get
\begin{equation}
\int_{a}^{b}\left\{\lambda(t) \,^{R}_{a,w}
D^{\alpha,\beta}\delta x(t) \right\}dt
=\int_{a}^{b}\left\{w(t)^2\delta x(t)\,^{R}D_{b,w}^{\alpha,\beta}
\frac{\lambda(t)}{w(t)^2}\right\}dt, 
\end{equation}
since $\lambda(t), x(t) \in L^2(a,b)$.
Then the equality (\ref{ogmented}) can be written as
\begin{equation}
\label{varogmented}
\delta J_{a}[x,\lambda,u]=\int_{a}^{b}\left\{\left[
\dfrac{\partial \mathcal{H}}{\partial x}-w(t)^2\,^{R}D_{b,w}^{\alpha,\beta}
\frac{\lambda(t)}{w(t)^2}\right]\delta x(t)
+ \left[\dfrac{\partial \mathcal{H}}{\partial \lambda}
- \,^{R}_{a,w}D^{\alpha,\beta}x(t)\right]\delta \lambda (t)
+ \dfrac{\partial \mathcal{H}}{\partial u}\delta u(t)\right\}dt.
\end{equation}
To obtain the necessary optimality conditions for the EWGFOCP,
the equality $\delta J_{a}=0$ should be verified
for all independent variations $\delta x(t)$, $\delta \lambda(t)$
and $\delta u(t)$, which requires that the corresponding coefficients must vanish. So,
\begin{eqnarray*}
w(t)^2\,^{R}D_{b,w}^{\alpha,\beta}\dfrac{\lambda(t)}{w(t)^2}
&=&\dfrac{\partial \mathcal{H}}{\partial x}\left(t, x(t),
\lambda (t),u(t)\right), \quad t\in[a,b],\\
\,^{R}_{a,w}D^{\alpha,\beta}x(t)
&=&\dfrac{\partial \mathcal{H}}{\partial \lambda}
\left(t, x(t),\lambda (t),u(t)\right), \quad t\in[a,b],\\
\dfrac{\partial \mathcal{H}}{\partial u}\left(t, x(t),\lambda (t),u(t)\right)
&=&0, \quad t\in[a,b],
\end{eqnarray*}
which completes the proof.
\end{proof}

\begin{Remark}
\label{rk}
As direct corollaries of Theorem~\ref{OC}, one can obtain
weak versions of Pontryagin's maximum principle for several
different fractional operators:
\begin{enumerate}
\item if $w(t) \equiv 1$ and $\beta=1$, then we obtain
a weak Caputo--Fabrizio fractional Pontryagin's maximum
principle with Hamiltonian system
\begin{equation}
\left\{
\begin{array}{ll}
\,^{R}_{a,1}D^{\alpha,1}x(t)=\dfrac{\partial \mathcal{H}}{\partial
\lambda}\left(t, x(t),\lambda (t),u(t)\right), & t\in[a,b],\\[0.3cm]
\,^{R}D_{b,1}^{\alpha,1}\lambda(t)=\dfrac{\partial
\mathcal{H}}{\partial x}\left(t, x(t),\lambda (t),u(t)\right), & t\in[a,b];
\end{array}
\right.
\end{equation}

\item if $w(t)\equiv 1$ and $\beta=\alpha$, then
we obtain a weak Atangana--Baleanu fractional Pontryagin's
maximum principle with Hamiltonian system
\begin{equation}
\left\{
\begin{array}{ll}
\,^{R}_{a,1}D^{\alpha,\alpha}x(t)=\dfrac{\partial \mathcal{H}}{\partial
\lambda}\left(t, x(t),\lambda (t),u(t)\right), & t\in[a,b],\\[0.3cm]
\,^{R}D_{b,1}^{\alpha,\alpha}\lambda(t)=\dfrac{\partial
\mathcal{H}}{\partial x}\left(t, x(t),\lambda (t),u(t)\right), & t\in[a,b];
\end{array}
\right.
\end{equation}

\item if $\beta=\alpha$, then we obtain a weak weighted Atangana--Baleanu
fractional Pontryagin's maximum principle with Hamiltonian system
\begin{equation}
\left\{
\begin{array}{ll}
\,^{R}_{a,w}D^{\alpha,\alpha}x(t)=\dfrac{\partial \mathcal{H}}{\partial
\lambda}\left(t, x(t),\lambda (t),u(t)\right), & t\in[a,b],\\[0.3cm]
w(t)^2\,^{R}D_{b,w}^{\alpha,\alpha}\dfrac{\lambda(t)}{w(t)^2}
=\dfrac{\partial \mathcal{H}}{\partial x}\left(t, x(t),
\lambda (t),u(t)\right), & t\in[a,b].
\end{array}
\right.
\end{equation}
\end{enumerate}
\end{Remark}


\section{Application: A Fractional Euler--Lagrange Equation}
\label{sec:04}

Let $L : [a,b] \times\mathbb{R}
\times \mathbb{R}\rightarrow\mathbb{R}$
be a differentiable function with respect to all its arguments
and consider the following fundamental problem
of the calculus of variations:
\begin{equation}
\label{eq:F}
J[x]=\int_a^b L\left( t,x(t), ^{R}D^{\alpha,\beta}_{a,w}x(t)\right)dt
\longrightarrow \min
\end{equation}
subject to the boundary conditions
\begin{equation}
\label{eq:BC}
x(a)=x_a, \quad x(b)=x_b.
\end{equation}

\begin{Theorem}[The weighted generalized fractional Euler--Lagrange equation]
\label{thm:SFE-Leq}
If $x$ is a minimizer of problem \eqref{eq:F}--\eqref{eq:BC},
then $x$ satisfies the following weighted generalized
fractional Euler--Lagrange equation:
\begin{equation}
\label{Euler-Lagrange}
\dfrac{\partial L}{\partial x} + w(t)^2\,^{R}D_{b,w}^{\alpha,\beta}
\dfrac{\dfrac{\partial L}{\partial x}}{w(t)^2}=0.
\end{equation}
\end{Theorem}

\begin{proof}
Set $u(t) := {^{R}_{a,w}D}^{\alpha,\beta}x(t)$.
The performance index of (\ref{eq:F}) becomes equivalent to the
objective functional
\begin{equation}
\label{eq:FF}
J[x,u]=\int_a^b L\left( t,x(t),u(t)\right)dt
\end{equation}
subject to a control system \eqref{P11} with
$f\left(t,x,u\right) = u$.
Applying our main Theorem~\ref{OC} to the current problem,
that is, using the Hamiltonian system \eqref{eq:HS}
and the stationary condition \eqref{eq:SC}, we have
\begin{equation}
\label{eq:prof:ELeq}
\begin{split}
\dfrac{\partial L}{\partial x} + w(t)^2\,^{R}D_{b,w}^{\alpha,\beta}
\dfrac{\dfrac{\partial L}{\partial u}}{w(t)^2}
&= \dfrac{\partial L}{\partial x}
+ w(t)^2\,^{R}D_{b,w}^{\alpha,\beta}
\dfrac{\dfrac{\partial H}{\partial u}-\lambda}{w(t)^2}
= \dfrac{\partial L}{\partial x} 
+ w(t)^2\,^{R}D_{b,w}^{\alpha,\beta}\dfrac{-\lambda}{w(t)^2}\\
&= \dfrac{\partial L}{\partial x} - \dfrac{\partial H}{\partial x}.
\end{split}
\end{equation}
Since
\begin{eqnarray*}
\mathcal{H}\left(t, x(t),\lambda (t),u(t)\right)
&=&\lambda (t) f\left(t,x(t),u(t)\right)+L\left(t,x(t),u(t)\right)
= \lambda (t) ^{R}_{a,w}D^{\alpha,\beta}x(t) +L\left(t,x(t),u(t)\right)\\
&=& \lambda (t) u(t) +L\left(t,x(t),u(t)\right),
\end{eqnarray*}
it follows from \eqref{eq:prof:ELeq} that
$$
\dfrac{\partial L}{\partial x} + w(t)^2\,^{R}D_{b,w}^{\alpha,\beta}
\dfrac{\dfrac{\partial L}{\partial x}}{w(t)^2}=0,
$$
which is the new Euler--Lagrange equation (\ref{Euler-Lagrange}).
\end{proof}	

\begin{Remark}
According to the cases listed in Remark~\ref{rk},
we can obtain several interesting Euler--Lagrange equations
as corollaries of Theorem~\ref{thm:SFE-Leq}, namely,
the Caputo--Fabrizio, Atangana--Baleanu,
and weighted Atangana--Baleanu fractional Euler--Lagrange equations.
\end{Remark}


\section{Conclusion}
\label{sec:06}

Using as mathematical tools some recent results of the literature
related to the fractional calculus theory with non singular kernels, 
we formulated a general fractional optimal control problem without
constraints on the values of the control functions, and proved
a necessary optimality condition of Pontryagin type.
As an application, a general fractional Euler--Lagrange equation
was obtained, showing the usefulness and interest of our main result.
Our results open many future possibilities for future research:
to extend our scalar problems and results to the vectorial case,
to generalize the integral cost functional to be given
as a fractional integral operator, etc. In principle,
all the developments done to the fractional calculus
of variations after the works of Riewe, see
\cite{MR3443073} and references therein, can now be
done for extended weighted generalized fractional
variational problems. The objective of our current study was 
to setup a general theoretical result related to a new formulation 
of the generalized weighted fractional optimal control problem. 
The idea was to contribute with theoretical tools to construct 
a generalized fractional theory of optimal control. As future work,
it would be interesting to develop numerical schemes for such kind of problems.


\section*{Acknowledgments}

This research was funded by Funda\c{c}\~{a}o para a Ci\^{e}ncia
e a Tecnologia (FCT), grant number UIDB/04106/2020 (CIDMA).
The authors are grateful to two anonymous reviewers and an editor
for valuable comments and suggestions, which helped them 
to improve the contents of the submitted version of the article.


\section*{Data availability statement}

Data sharing is not applicable to this article 
as no new data were created or analyzed in this study.



\end{document}